\documentclass[12pt]{amsart}
\usepackage{amsmath}
\usepackage{amssymb}
\usepackage{amsfonts}
\usepackage{amsthm}
\usepackage{mathrsfs}
\usepackage[all]{xy}

\newcommand{\be}{\begin{equation}}
\newcommand{\ee}{\end{equation}}
\newtheorem{theorem}{Theorem}[section]
\newtheorem{lemma}[theorem]{Lemma}

\newtheorem{proposition}[theorem]{Proposition}

\theoremstyle{definition}
\newtheorem{definition}[theorem]{Definition}

\theoremstyle{remark}

\numberwithin{equation}{section}

\begin{document}

\title{Principal realization of twisted Yangian $Y(\mathfrak{g}_{N})$}
\author{Naihuan Jing, Ming Liu$^*$}
\address{NJ \& ML: School of Science, South China University of Technology,
Guangzhou 510640, China}
\address{NJ: Department of Mathematics, North Carolina State University, Raleigh, NC 27695, USA}

\thanks{{\scriptsize
\hskip -0.4 true cm MSC (2010): Primary: 17B37; Secondary: 17B65.
\newline Keywords: Yangians, principal realization.\\
$*$Corresponding author.
}}

\maketitle

\begin{abstract}  We give the principal realization of the twisted Yangians
of orthogonal and symplectic types. The new bases are interpreted in terms of
discrete Fourier transform over the cyclic group $\mathbb Z_N$.
\end{abstract}

\section{Introduction}

Let $\mathfrak g$ be a simple complex Lie algebra.
As one of the two important classes of quantum groups
associated to $\mathfrak g$, the Yangian $Y(\mathfrak g)$ was introduced by Drinfeld
in the study of the Yang-Baxter equation \cite{D1, D2, D3, CP}. Other versions of
Yangians were given by Olshanski \cite{O1} in connection with classical groups.
All these Yangian algebras play important
roles in conformal field theory, combinatorics and representation theory
(see \cite{MNO} for a beautiful
survey and the monograph \cite{M} for recent developments).
For other presentations of Yangian algebras, see also \cite{BK}.

The Yangian $Y(\mathfrak{
gl}_N)$ of the general liner algebra captures and unifies, on a higher theoretic ground, many far reaching
aspects of invariant theory and combinatorial theory \cite{M2, GM}.
Motivated by the principal realization of the affine Kac-Moody Lie algebras, the authors
in \cite{BGJ} introduced the corresponding set of generators for the Yangian $Y(\mathfrak{
gl}_N)$ and show that the new basis is useful in studying representations of Yangians.
Roughly speaking, the idea is based on replacing the Cartan-Weyl basis by the
the Toeplitz basis in the general linear Lie algebra $\frak{gl}_N$, this enables one to
get new presentations of the Yangian.

We will study the principal generators for the Olshanski Yangian
algebras $Y(\mathfrak{so}_N)$ and
$Y(\mathfrak{sp}_N)$ in this letter and give their main properties. A new feature is that
all principal generators are actually discrete Fourier
transform of certain sequences defined by the Yangians
over the cyclic group $\mathbb Z_{N}$. It seems that the twisted cases
are also related to the Fourier transform over the abelian group $\mathbb Z_N\times \mathbb Z_N$.
Under discrete Fourier transform, several interesting properties are formulated
in the same pattern for the principal generators in the cases of the orthogonal and symplectic Yangians.

The paper is organized as follows. First in section two we recall the principal generators
for type $A$ and formulate the principal generators as Fourier transform. Several new results are
proved for later usage.
Section three discusses Olshanski twisted Yangian algebras. Section four
gives the principal realizations for the twisted Yangian algebras
of orthogonal and symplectic types.

\section{Principal realization of Yangian algebra $Y(\mathfrak{gl}_N)$}

\vskip 0.2in
   In this section, we first recall the principal realization for $Y(\mathfrak{gl}_N)$
   and then derive new relations using discrete Fourier transform for our later purpose.

\subsection{Principal basis of $\mathfrak{gl}_N$}

\vskip .1in

Let $\mathfrak{g}=\mathfrak{gl}_N$ be the Lie algebra of $N\times N$ complex matrices. The standard
Cartan-Weyl basis consists of matrices $E_{ij}$, where $i,j\in \mathbb{Z}_N=\{0,1...N-1\}$. For our
purpose we will make full use of the additive structure of the index set $\mathbb{Z}_N$.

The cyclic element in $\mathfrak{gl}_N$ is $$E=\sum_{i\in\mathbb{Z}_N}E_{i,i+1}$$
and its centralizer $C(E)=\bigoplus_{k\in\mathbb{Z}_N}\mathbb{C}E^k$
is a Cartan subalgebra of $\mathfrak{gl}_N$ called the principal Cartan subalgebra.
With respect to this Cartan subalgebra, the principal root space
decomposition is given as follows:

\begin{equation}\label{fomula2.1}
\mathfrak{gl}_N=\bigoplus_{i,j\in \mathbb{Z}_N} A_{ij},
A_{ij}=\sum_{k\in \mathbb{Z}_N}\omega^{ki}E_{k,k+j},
\end{equation}
where $\omega=e^{\frac{\mathfrak{i}2\pi}{N}}$.

Let $G$ be a finite abelian group with irreducible characters $\chi_i$, $(i=1, \ldots, |G|)$, the discrete
Fourier transform of the function $f(g)$ on $G$ is another function on $G^*=\{\chi_i|i=1, \ldots, |G|\}$ defined by
\be
\mathfrak{F}(f)(\chi)=\sum_{h\in G}\chi(h)f(h).
\ee

In the case of cyclic group $G=\mathbb Z_N$, $\mathbb Z_N^*=\{\chi^i|i=0, \ldots, N-1\}\simeq Z_N$,
and $\chi_i(j)=\omega^{ij}$. Therefore the discrete Fourier transform of the function $f$ on $\mathbb Z_N$ is
\be
\mathfrak{F}(f)(\chi_i)=\sum_{j=0}^{N-1}\omega^{ij}f(j).
\ee

The inverse Fourier transform is given by
\be
\mathfrak{F^{-1}}(g)(i)=\frac1N\sum_{j=0}^{N-1}\omega^{-ij}g(j).
\ee

Fix $j$, denote the finite sequence $\{E_{k,k+j}\}$ by $\{\epsilon_j\}$,
where $k$ runs over $\mathbb{Z}_N$, i.e., $\epsilon_j(k)=E_{k,k+j}$. Then the principal basis elements $A_{ij}$ are actually the $i$th term of the discrete Fourier transform of the sequence $\{\epsilon_j\}$, and
formula (\ref{fomula2.1}) takes the following new form:
$$A_{ij}=\mathfrak{F}(\{\epsilon_j(k)\})(i).$$
\vskip .2in
The algebraic structure of principal basis $A_{ij}$ is given by:

$$A_{ij}A_{kl}=\omega^{jk}A_{i+k,j+l}.$$

 Under the standard inner product $(x|y)=tr(xy)$, we have

 $$(A_{ij}|A_{kl})=tr(A_{ij}A_{kl})=N\omega^{-ij}\delta_{i,-k}\delta_{j,-l}, $$
and the dual principal basis is $\{\frac{\omega^{ij}}{N}A_{-i,-j}\}$.

 \subsection{Yangian $Y(\mathfrak{{gl}}_N)$}

 \begin{definition}

 The Yangian algebra $Y(\mathfrak{gl}_N)$ is an unital associative algebra
 with generators $t_{ij}^{(r)}$ ($i,j\in \{1,2...N\},r\in \mathbb{Z}_{+}$)
 subject to the relations:
\begin{align}\label{definingrelation1}
[t^{(r+1)}_{ij},t^{(s)}_{kl}]-[t^{(r)}_{ij},t^{(s+1)}_{kl}]
=t^{(r)}_{kj}t^{(s)}_{il}-t^{(s)}_{kj}t^{(r)}_{il},
\end{align}
where $t_{ij}^{(0)}=\delta_{ij}$.
\end{definition}
Its matrix presentation is given in terms of the rational Yang-Baxter $R$-matrix.
 Let $u$ be a formal variable and let
 \begin{align}\label{R-matrix}
 R(u)=1-\frac{P}{u}\in End(\mathbb{C}^{N})\otimes End( \mathbb{C}^{N})[[u^{-1}]],
 \end{align}
where $P$ is the permutation matrix: $P(u\otimes v)=v\otimes u$ for any $u, v\in \mathbb{C}^N$.
The matrix $R(u)$ satisfies the quantum Yang-Baxter equation:
 $$R_{12}(u)R_{13}(u+v)R_{23}(v)=R_{23}(v)R_{13}(u+v)R_{12}(u).$$

Set $$T(u)=\sum_{i,j}t_{ij}(u)\otimes E_{ij}\in
Y(\mathfrak{gl}_N)[[u^{-1}]]\otimes End(\mathbb{C}^{N})$$ where
 $$t_{ij}(u)=\delta_{ij}+\sum_{k=1}^{\infty}t_{ij}^{(k)}u^{-k}\in Y(\mathfrak{gl}_N)[[u^{-1}]].$$
The following well-known result (see \cite{M}) gives the FRT formulation \cite{FRT}
of the Yangian algebra.

 \begin{proposition} The defining relations of Yangian can be written compactly
 as
 $$R(u-v)T_{1}(u)T_{2}(v)=T_{2}(v)T_{1}(u)R(u-v).$$
\end{proposition}
\subsection{The principal realization of $Y(\mathfrak{gl}_N)$}
 In \cite{BGJ} a new set of generators, the principal generators, $x_{ij}^{(k)}$ are introduced. Here the indices $i,j\in \mathbb{Z}_{N}$, $k \in \mathbb{N}$. Let $x_{ij}(u)$ be the generating series:
 $$x_{ij}(u)=\sum_{n=0}^{\infty}x_{ij}^{(n)}u^{-n},$$
 where $x_{ij}^{(0)}=\delta_{i,0}\delta_{j,0}$. Then
 \begin{equation}\label{prin1}
 x_{ij}(u)=\sum_{k\in\mathbb Z_N}\frac{\omega^{-ki}}{N}t_{k,j+k}(u).
 \end{equation}
This can be viewed as the inverse of the Fourier transform on the finite sequence $\{t_{k,k+j}(u)\}$ where
the variable $k\in \mathbb{Z}_N$. Then the formula (\ref{prin1}) can be rewritten as
$$x_{ij}(u)=\mathfrak{F}^{-1}(\{t_{k,k+j}(u)\})(i)$$
 Rewriting the T-matrix $T(u)$ by using the principal basis
 of $\mathfrak{gl}_N$ and $x_{ij}(u)$ as follows
 $$T(u)=\sum_{k,l\in \mathbb{Z}_N}x_{kl}(u)\otimes A_{kl},$$
 we obtain the principal realization of $Y(\mathfrak{gl}_N)$ as follows.
 \vskip.2in

 \begin{proposition}\cite{BGJ}\label{BGJ}
 The principal generators $x_{ij}^{(k)}$ of Yangian $Y(\mathfrak{gl}_N)$ satisfy the following relations:
 \begin{align*}
 &[x_{ij}(u),x_{kl}(v)]=\frac{1}{u-v}(\sum_{a,b}\frac{\omega^{ib-bk-ab}}{N}x_{i-a,j-b}(u)x_{k+a,l+b}(v)\\
 &~~~~~~~~~~~~~~~~-\sum_{a,b}\frac{\omega^{ja-al-ab}}{N}x_{k+a,l+b}(v)x_{i-a,j-b}(u))
 \end{align*}
 \end{proposition}

 We can simplify the commutation relations and get a new compact formula as follows:
\begin{theorem}\label{new}
 The principal generators $x_{ij}^{(k)}$ of Yangian $Y(\mathfrak{gl}_N)$ satisfy the following relations:
 \begin{align*}
 &(u-v)[x_{ij}(u),x_{kl}(v)]=\frac1N\sum_{a,b}\omega^{-ab}(x_{k+a,j+b}(u)x_{i-a,j-b}(v)-x_{k+a,j+b}(v)x_{i-a,j-b}(u)),
 \end{align*}
 where $a,b$ run through the group $\mathbb Z_N$.
\end{theorem}
It is easy to see that Proposition \ref{BGJ} is a consequence of Theorem \ref{new}.
 
\section{Twisted Yangian $Y(\mathfrak{so}_N)$ and $Y(\mathfrak{sp}_N)$}
\vskip.2in
We first describe the structure of the Oshanski twisted Yangian algebras associated to $\mathfrak{so}_N$ and $\mathfrak{sp}_N$ in this section.

\subsection{The Lie algebra $\mathfrak{so}_N$ and $\mathfrak{sp}_N$}

\vskip.2in
We will consider simultaneously both $\mathfrak{so}_N$ and $\mathfrak{sp}_N$.
In the following let the index set $\mathbb Z_N=\{0, \ldots, N-1\}$ for matrices in $Mat(N)$.
Let $A\mapsto A^{t}$ denote the transposition of $Mat(N)$ defined by
$$E^{t}_{ij}=\theta_{i}\theta_{j}E_{N-1-j,N-1-i},$$
where $i,j \in \mathbb{Z}_N$. The scalar $\theta_i$ is defined according to
two cases as follows. For the symmetric case,
\begin{align*}
\theta_{i}=1, ~i=0,1, \ldots, N-1,
\end{align*}
and for the alternating or antisymmetric case with $N=2n$,
\begin{align*}
\theta_{i}=\left\{
                     \begin{array}{ll}
                       -1, & \hbox{$i=0,1, \ldots,n-1$;} \\
                       1, & \hbox{$i=n,n+1, \ldots, 2n-1$.}
                     \end{array}
                   \right.
\end{align*}

Introduce the following elements of the Lie algebra $\mathfrak{gl}_N$:
\begin{align*}
F_{ij}=E_{ij}-E^{t}_{ij}=E_{ij}-\theta_{i}\theta_{j}E_{N-1-j,N-1-i},
\end{align*}
then the Lie subalgebra 
spanned by $F_{ij}$ is isomorphic to $\mathfrak{so}_N$ in the symmetric case and to $\mathfrak{sp}_N$ in the alternating case. The resulting Lie algebra will be denoted by $\mathfrak{g}_N$. Thus,
\begin{align*}
\mathfrak{g}_N=\mathfrak{so}_N  ~or ~ \mathfrak{sp}_N,
\end{align*}
where the latter case $N$ is supposed to be even.
Corresponding to the principal basis of $\mathfrak{gl}_N$, we can derive the following simple result.

\begin{proposition}\label{P:dualmatirx}
 The subalgebra of $\mathfrak{gl}_N$ spanned by the elements $B_{ij}=A_{ij}-A_{ij}^{t}$ $i,j\in \mathbb{Z}_N$
 is isomorphic to $\mathfrak{g}_N$, where \\
$$A^{t}_{ij}=\frac{\omega^{-i(1+j)}}{N}\sum_{k,l\in\mathbb{Z}_N}\theta_k\theta_{k+j} \omega^{-k(i+l)}A_{lj}.$$
In particular, $A_{ij}^{t}=\omega^{-i(1+j)}A_{-i,j}$ in the symmetric case.
\end{proposition}

\begin{proof}
It follows from Equation (\ref{fomula2.1}) that
\begin{align*}
&A_{ij}^{t}= \sum_{k\in \mathbb{Z}_N}\omega^{ki}E^{t}_{k,k+j}\\
&~~~~~~~~~~=\sum_{k\in \mathbb{Z}_N}\omega^{ki}\theta_k\theta_{k+j}E_{N-1-k-j,N-1-k}\\
&~~~~~~~~~~=\sum_{k\in \mathbb{Z}_N}\omega^{(N-1-j-k)i}\theta_{N-1-k-j}\theta_{N-1-k}E_{k,k+j}\\
&~~~~~~~~~~=\omega^{-i(1+j)}\sum_{k\in \mathbb{Z}_N}\omega^{-ki}\theta_{k+j}\theta_{k}E_{k,k+j},
\end{align*}
where we used $\theta_{N-1-k}=-\theta_{k}$ and $\omega^N=1$ in the last equation.
It follows from equation (\ref{fomula2.1}) and inverting the Fourier transform that
\begin{align*}
E_{k,k+j}=\frac{1}{N}\sum_{l\in \mathbb{Z}_N}\omega^{-kl}A_{lj}.
\end{align*}
Therefore we have
\begin{align*}
A_{ij}^{t}= \frac{\omega^{-i(1+j)}}{N}\sum_{k,l\in \mathbb{Z}_N}\theta_{k+j}\theta_{k}\omega^{-k(i+l)}A_{lj}.
\end{align*}
In particular, $\theta_k=\theta_{k+j}=1$ in the symmetric case and
\begin{align*}
&\sum_{l\in \mathbb{Z}_N}\omega^{-kl}A_{lj}=\sum_{m,l\in \mathbb{Z}_N}\omega^{-kl}\omega^{ml}E_{m,m+j}\\
&~~~~~~~~~~~~~~~~~~~~~~~~~~~~~~~~~~~~~~~~~~=\sum_{m\in \mathbb{Z}_N}(\sum_{l\in \mathbb{Z}_N}\omega^{(-k+m)l})E_{m,m+j}\\
&~~~~~~~~~~~~~~~~~~~~~~~~~~~~~~~~~~~~~~~~~~=\sum_{m\in \mathbb{Z}_N}N\delta_{m,k}E_{m,m+j}=NE_{k,k+j}.
\end{align*}
So in the case $\mathfrak{g}_N=\mathfrak{so}_N$
$$A^{t}_{ij}=\omega^{-i(1+j)}A_{-i,j}.$$
\end{proof}

Just as in the case of $\mathfrak{gl}_N$, we can also interpret the generators $B_{ij}'s$ using Fourier transform.
$$B_{ij}=\mathfrak{F}(\{F_{k,k+j}\})(i).$$

\subsection{Twisted Yangians $Y(\mathfrak{g}_N)$}
\vskip.2in
\begin{definition}
The twisted Yangian corresponding to $\mathfrak{g}_N$ is a unital
associative algebra with generators $s_{ij}^{(1)},s_{ij}^{(2)},...$, where $i,j\in \mathbb{Z}_N$,
and the defining relations are given in terms of generating series
$$s_{ij}(u)=\delta_{ij}+s_{ij}^{(1)}u^{-1}+s_{ij}^{(2)}u^{-2}+...$$
as follows.
\begin{align*}
&(u^2-v^2)[s_{ij}(u),s_{kl}(v)]=(u+v)(s_{kj}(u)s_{il}(v)-s_{kj}(v)s_{il}(u))\\
&~~~~~~~~~~~~~~~~~~~~~~~~~~~~~~-(u-v)(\theta_k\theta_{j'}s_{ik'}(u)s_{j'l}(v)-\theta_i\theta_{l'}s_{ki'}(v)s_{l'j}(u))\\
&~~~~~~~~~~~~~~~~~~~~~~~~~~~~~~+\theta_i\theta_{j'}(s_{ki'}(u)s_{j'l}(v)-s_{ki'}(v)s_{j'l}(u)),
\end{align*}
and
\begin{align}
\theta_i\theta_js_{j'i'}(-u)=s_{ij}(u)\pm\frac{s_{ij}(u)-s_{ij}(-u)}{2},
\end{align}
where $i'=N-1-i$. Whenever the double sign $\pm$ or $\mp$ occurs, the upper sign corresponds to the $\mathfrak{so}_N$ case
and the lower sign to the $\mathfrak{sp}_N$ case.
\end{definition}
These relations can also be given in an equivalent matrix form. For this purpose we define the partial transpose
$R^{t}(u)$ for the Yang's R-matrix (\ref{R-matrix}) by

\begin{align}
R^{t}(u)=R(u)^{t_1}=1-Qu^{-1}, \quad Q=\sum_{i,j\in \mathbb{Z}_N}E^t_{ij}\otimes E_{ji}.
\end{align}
Furthermore, we regard $S(u)$ as an element of the algebra $Y(\mathfrak{g}_N)\otimes End\mathbb{C}^N$
given by
\begin{align*}
S(u)=\sum_{i, j\in \mathbb{Z}_N} s_{ij}(u)\otimes E_{ij}.
\end{align*}

Then the twisted Yangian can be characterized by the following relations \cite{M}:
\begin{equation}\label{reflectionrelation}
R(u-v)S_{1}(u)R^{t}(-u-v)S_{2}(v)=S_{2}(v)R^{t}(-u-v)S_{1}(u)R(u-v),
\end{equation}
and the symmetric relation:
\begin{equation}\label{symmetricrelation}
S^{t}(-u)=S(u)\pm \frac{S(u)-S(-u)}{2u}.
\end{equation}

The following relations between twisted Yangian and classical Lie
algebras will be useful.

\begin{proposition}\cite{M}\label{proposition3.3}
The assignment
$$s_{ij}(u)\mapsto \delta_{ij}+(u\pm\frac{1}{2})^{-1}F_{ij}$$ defines
a homomorphism between $Y(\mathfrak{g}_N)$ and $U(\mathfrak{g}_N)$. Moreover, the assignment
$$F_{ij}\mapsto s^{(1)}_{ij} $$
defines an embedding $U(\mathfrak{g}_N)\hookrightarrow Y(\mathfrak{g}_N)$.
\end{proposition}

\section{Principal realization of $Y(\mathfrak{g}_N)$}

In this section we give the principal realization for twisted Yangian $Y(\mathfrak{g}_N)$ analogous to the $Y(\mathfrak{gl}_N)$ case. As before the generators are certain Fourier coefficients.

We start by recalling a well-known result in linear algebra.
Let $\{e_i\}$ and $\{e^i\}$ be a pair of dual bases of simple Lie
algebra $\mathfrak{g}$, then the rational $r$-matrix can be expressed as
follows (cf. \cite{BGJ}):
\begin{equation}\label{rmatrix}
r=\sum e_i\otimes e^i.
\end{equation}
Moreover, this expression is independent of the choice of the dual bases.

Using the principal basis of $\mathfrak{gl}_N$, we get the following result.
\vskip.2in
\begin{lemma}
The permutation matrix $P$ can be
written as
\begin{equation}
P=\sum_{i,j}E_{ij}\otimes E_{ji}=\sum_{k,l\in \mathbb{Z}_N}\frac{\omega^{kl}}{N}A_{kl}\otimes A_{-k,-l}.
\end{equation}
The partial transposition of $P$ can be written as
\begin{align*}
Q=P^{t_1}=\sum_{k,l\in \mathbb{Z}_N}\frac{\omega^{kl}}{N}A_{kl}^{t}\otimes A_{-k,-l}=\sum_{k,l,a,b}\frac{\omega^{-a(k+b)-k}}{N^2}\theta_{a,a+l}A_{bl}\otimes A_{-k,-l}
\end{align*}
In particular, when $\mathfrak{g}_N=\mathfrak{so}_N$
\begin{equation}\label{equation4.3}
Q=P^{t_1}=\sum_{k,l\in \mathbb{Z}_N}\frac{\omega^{kl}}{N}A_{kl}^{t}\otimes A_{-k,-l}=
\sum_{k,l\in \mathbb{Z}_N}\frac{\omega^{-k}}{N}A_{-k,l}\otimes A_{-k,-l}.
\end{equation}
\end{lemma}

\begin{proof} Note that $\{A_{kl}\}$ and $\{\frac{\omega^{kl}}{N}A_{-k,-l}\}$
are dual principal basis of $\mathfrak{gl}_N$. Invoking (\ref{rmatrix}) we have
$$P=\sum_{k,l\in \mathbb{Z}_N}\frac{\omega^{kl}}{N}A_{kl}\otimes A_{-k,-l}.$$
Using Proposition \ref{P:dualmatirx}, we get
\begin{align*}
&Q=\sum_{k,l\in \mathbb{Z}_N}\frac{\omega^{kl}}{N}A_{kl}^{t}\otimes A_{-k,-l}\\
&~~=\sum_{k,l,a,b}\frac{\omega^{-a(k+b)-k}}{N^2}\theta_{a,a+l}A_{bl}\otimes A_{-k,-l}.
\end{align*}
And in the case of $\mathfrak{g}_N=\mathfrak{so}_N$,
\begin{align*}
&Q=\sum_{k,l\in \mathbb{Z}_N}\frac{\omega^{kl}}{N}A_{kl}^{t}\otimes A_{-k,-l}\\
&~~=\sum_{k,l\in \mathbb{Z}_N}\frac{\omega^{-k}}{N}A_{-k,l}\otimes A_{-k,-l}.
\end{align*}
\end{proof}

We now introduce a new set of generators $y^{(r)}_{ij} $ of twisted Yangian $Y(\mathfrak{g}_N)$, where $i,j\in \mathbb{Z}_N$,
$r\in\mathbb Z_+.$
 Rewrite the matrix of generators $S(u)$ as follows:
 \begin{equation}
 S(u)=\sum_{i,j}s_{ij}(u)\otimes E_{ij}=\sum_{k,l\in \mathbb{Z}_N}y_{kl}(u)\otimes A_{kl},
 \end{equation}
 where $y_{kl}(u)'$s are the generating series defined by:
 $$y_{kl}(u)=\sum_{r=0}^{\infty}y^{(r)}_{kl}u^{-r},$$
 and $y^{(0)}_{kl}=\delta_{k,0}\delta_{l,0}$.
 \vskip.1in

\vskip.2in
\begin{theorem} \label{thm1} The principal generators
$y_{ij}^{(r)}$ of the Yangian $Y(\mathfrak{g}_N)$ satisfy the
following relations: 
\begin{align*}
&(u^2-v^2)[y_{ij}(u),y_{kl}(v)]=\\
&\frac{u+v}{N}\sum_{a,b\in\mathbb{Z}_N}(\omega^{b(i-k-a)}y_{i-a,j-b}(u)y_{k+a,l+b}(v)-\omega^{a(j-l-b)}y_{k+a,l+b}(v)y_{i-a,j-b}(u))\\
&+\frac{u-v}{N^2}\sum_{a,b,a',b'\in
\mathbb{Z}_N}\theta_{a'}\theta_{a'+b}\omega^{-a-a'(a+b')+b'(j-b)-b(k+a))}
y_{i-b',j-b}(u)y_{k+a,l+b}(v)\\
&-\frac{u-v}{N^2}\sum_{a,b,a',b'\in
\mathbb{Z}_N}\theta_{a'}\theta_{a'+b}\omega^{-a-a'(a+b')+b(i-b')-a(l+b))}y_{k+a,l+b}(v)y_{i-b',j-b}(v)\\
&+\frac{1}{N^3}\sum_{a,b,a',b',c,d\in\mathbb{Z}_N}\theta_{a'}\theta_{a'+b}\omega^{-a-a'(a+b')+b'(j-b-d)+b(k+a+c)+d(i-k-c)}
y_{i-c-b',j-d-b}(u)y_{k+a+c,l+b+d}(v)\\
&-\frac{1}{N^3}\sum_{a,b,a',b',c,d\in
\mathbb{Z}_N}\theta_{a'}\theta_{a'+b}\omega^{-a-a'(a+b')+b(i-a-b')+a(l+b+d)+c(j-l-d)}y_{k+a+c,l+b+d}(v)y_{i-c-b',j-d-b}(u)
\end{align*}
and the symmetric relation:
$$\frac{1}{N}\sum_{k,l}\theta_{k}\theta_{k+j}\omega^{-l(1+j)-k(i+l)}y_{l,j}(-u)
=y_{ij}(u)\pm\frac{y_{ij}(u)-y_{ij}(-u)}{2u}.$$
In particular, when $\mathfrak{g}_N=\mathfrak{so}_N$ the defining relations can be reduced to the following relations:
\begin{align*}
&(u^2-v^2)[y_{ij}(u),y_{kl}(v)]=\\
&\frac{(u+v)}{N}\sum_{a,b\in\mathbb{Z}_N}(\omega^{b(i-k-a)}y_{i-a,j-b}(u)y_{k+a,l+b}(v)-\omega^{a(j-l-b)}y_{k+a,l+b}(v)y_{i-a,j-b}(u))\\
&+\frac{(u-v)}{N}\sum_{a,b\in\mathbb{Z}_N}(\omega^{-a+bi-la}y_{k+a,l+b}(u)y_{i+a,j-b}(v)-\omega^{-a-aj-bk}y_{i+a,j-b}(u)y_{k+a,l+b}(v))\\
&+\frac{1}{N^2}\sum_{a,b,a',b'\in\mathbb{Z}_N}\omega^{-a-aj+b'(a-k+i-a')+b(k+a')}y_{i+a-a',j-b-b'}(u)y_{k+a+a',l+b+b'}(v)\\
&-\frac{1}{N^2}\sum_{a,b,a',b'\in
\mathbb{Z}_N}\omega^{-a+bi+a'(j-b-b'-l)-a(l+b')}y_{k+a+a',l+b+b'}(v)y_{i+a-a',j-b-b'}(u)
\end{align*}
and the symmetric relation:
\begin{equation*}
\omega^{i(1+j)}y_{-i,j}(-u)=y_{ij}(u)+\frac{1}{2u}(y_{ij}(u)-y_{ij}(-u)).
\end{equation*}
\end{theorem}
Here we only prove the case $\mathfrak{g}_N=\mathfrak{so}_N$ as the other case is similar.
In order to prove the theorem, we need the following lemma:
\begin{lemma}\label{lemma4.3}
We have the following equations for $\mathfrak{g}_N=\mathfrak{so}_N$:
\begin{equation}
\begin{aligned}
PS_{1}(u)S_2(v)-S_{2}(v)S_{1}(u)P=\frac{1}{N}\sum_{a,b,i,j,k,l\in\mathbb{Z}_N}(\omega^{b(i-k-a)}y_{i-a,j-b}(u)y_{k+a,l+b}(v)\\
-\omega^{a(j-l-b)}y_{k+a,l+b}(v)y_{i-a,j-b}(u))A_{ij}\otimes A_{kl}.
\end{aligned}
\end{equation}

\begin{equation}\label{equation4.6}
\begin{aligned}
S_{2}(v)QS_{1}(u)-S_{1}(u)QS_{2}(v)
=\frac{1}{N}\sum_{a,b,i,j,k,l\in\mathbb{Z}_N}(\omega^{-a+bi-la}y_{k+a,l+b}(u)y_{i+a,j-b}(v)\\
-\omega^{-a-aj-bk}y_{i+a,j-b}(u)y_{k+a,l+b}(v))A_{ij}\otimes A_{kl}.
\end{aligned}
\end{equation}
\begin{equation}
\begin{aligned}
&PS_{1}(u)QS_{2}(v)-S_{2}(v)QS_{1}(u)P=\\
&\frac{1}{N^2}\sum_{a,b,a',b',i,j,k,l\in\mathbb{Z}_N}\omega^{-a-aj+b'(a-k+i-a')+b(k+a')}y_{i+a-a',j-b-b'}(u)y_{k+a+a',l+b+b'}(v) A_{ij}\otimes A_{kl}-\\
&\frac{1}{N^2}\sum_{a,b,a',b',i,j,k,l\in\mathbb{Z}_N}\omega^{-a+bi+a'(j-b-b'-l)-a(l+b')}y_{k+a+a',l+b+b'}(v)y_{i+a-a',j-b-b'}(u) A_{ij}\otimes A_{kl}.
\end{aligned}
\end{equation}

\end{lemma}
\begin{proof}
Here we just check the equation (\ref{equation4.6}), and the other two equations can be proved by the same method.
Using the principal decomposition of the operator $Q$ (see formula (\ref{equation4.3})), we have
\begin{align*}
&S_{2}(v)QS_{1}(u)=(\sum_{k,l\in \mathbb{Z}_N}y_{kl}(v)1\otimes A_{kl})(\frac{1}{N}\sum_{a,b\in \mathbb{Z}_N}\omega^{-a}A_{-a,b}\otimes A_{-a,-b})
(\sum_{i,j\in \mathbb{Z}_N}y_{ij}(u)A_{ij}\otimes 1)\\
&~~~~~~~~~~~~~~~~~=\frac{1}{N}\sum_{a,b,i,j,k,l\in \mathbb{Z}_N}\omega^{-a}y_{kl}(v)y_{ij}(u)A_{-a,b}A_{ij}\otimes A_{kl}A_{-a,-b}\\
&~~~~~~~~~~~~~~~~~=\frac{1}{N}\sum_{a,b,i,j,k,l\in \mathbb{Z}_N}\omega^{-a+bi-la}y_{kl}(v)y_{ij}(u)A_{-a+i,b+j}\otimes A_{k-a,l-b}\\
&~~~~~~~~~~~~~~~~~~(Let ~~i'=-a+i,j'=b+j,k'=k-a,l'=l-b)\\
&~~~~~~~~~~~~~~~~~=\frac{1}{N}\sum_{a,b,i',j',k',l'\in \mathbb{Z}_N}\omega^{-a+bi'-l'a}y_{k'+a,l'+b}(v)y_{i'+a,j'-b}(u)A_{i'j'}\otimes A_{k'l'}\\
&~~~~~~~~~~~~~~~~~=\frac{1}{N}\sum_{a,b,i,j,k,l\in \mathbb{Z}_N}\omega^{-a+bi-la}y_{k+a,l+b}(v)y_{i+a,j-b}(u)A_{ij}\otimes A_{kl}.\\
\end{align*}
Similarly, we can get
$$S_{1}(u)QS_{2}(v)=\frac{1}{N}\sum_{a,b,i,j,k,l\in \mathbb{Z}_N}\omega^{-a-aj-bk}y_{i+a,j-b}(u)y_{k+a,l+b}(v))A_{ij}\otimes A_{kl}, $$
from which one gets Equation (\ref{equation4.6}).
\end{proof}
Now we prove Theorem \ref{thm1} (the case $\mathfrak{g}_N=\mathfrak{so}_N$) using the above lemma.
\begin{proof}
From the equation (\ref{reflectionrelation}), it follows that
\begin{align*}
&(u^2-v^2)(S_{1}(u)S_{2}(v)-S_{2}(v)S_{1}(u))=(u+v)(PS_{1}(u)S_2(v)-S_{2}(v)S_{1}(u)P)\\
&~~~~~~~~~~~~~~~~~~~~~~~~~~~~~~~~~+(u-v)(S_{2}(v)QS_{1}(u)-S_{1}(u)QS_{2}(v))\\
&~~~~~~~~~~~~~~~~~~~~~~~~~~~~~~~~~+(PS_{1}(u)QS_{2}(v)-S_{2}(v)QS_{1}(u)P).\\
\end{align*}
Then one derives the reflection relation in the Theorem \ref{thm1} by using the equations
in lemma \ref{lemma4.3}.
\vskip0.1in
Next we check the symmetric relation.
It follows from definition that the symmetric relation is
$$S^{t}(-u)=S(u)+\frac{S(u)-S(-u)}{2u}.$$
Using the principal presentation of $S(u)$ we have:
\begin{align*}
&S(u)=\sum_{i,j\in \mathbb{Z}_N}y_{ij}(u)\otimes A_{ij},\\
&S^{t}(-u)=\sum_{i,j\in \mathbb{Z}_N}y_{ij}(-u)\otimes A^{t}_{ij}\\
&~~~~~~~~~=\sum_{i,j\in \mathbb{Z}_N}\omega^{-i(1+j)}y_{ij}(-u)\otimes A_{-i,j}\\
\end{align*}
Then we can rewrite the symmetric relation as following:
\begin{align*}
\sum_{i,j\in \mathbb{Z}_N}\omega^{-i(1+j)}y_{ij}(-u)\otimes A_{-i,j}=\sum_{i,j\in \mathbb{Z}_N}y_{ij}(u)\otimes A_{ij}
+\frac{1}{2u}\sum_{i,j\in \mathbb{Z}_N}(y_{ij}(u)-y_{ij}(-u))\otimes A_{ij}.
\end{align*}
Therefore
\begin{equation*}
\omega^{i(1+j)}y_{-i,j}(-u)=y_{ij}(u)+\frac{1}{2u}(y_{ij}(u)-y_{ij}(-u)),
\end{equation*}
which is just the symmetric relation in the theorem.
\end{proof}
\vskip.2in

\begin{theorem}\label{thm2}
 The mapping $s_{ij}(u)\mapsto \sum_{k\in \mathbb{Z}_N}\omega^{ik}y_{k,j-i}(u)$ defines an
isomorphism of the two presentations of $Y(\mathfrak{g}_N)$. The inverse mapping is given by
$$y_{kl}(u)\mapsto \sum_{i\in \mathbb{Z}_N}\frac{\omega^{-ki}}{N}s_{i,i+l}(u)$$
\end{theorem}

 \begin{proof} This can be quickly shown by the inverse Fourier transform. For completeness we give
 another proof. From the principal realizations of twisted
Yangian $Y(\mathfrak{g}_N)$, we have
\begin{equation*}
S(u)=
\sum_{k,l\in \mathbb{Z}_N}y_{kl}(u)\otimes A_{kl},
\end{equation*}
where
$$A_{kl}=\sum_{a\in \mathbb{\mathbb{Z}}_N}\omega^{ak}E_{a,a+l}.$$
Plugging back into the equation, one has
\begin{align*}
S(u)&=
\sum_{k,l,a\in \mathbb{Z}_N}\omega^{ak}y_{kl}(u)\otimes E_{a,a+l}\\
&=\sum_{i,j,k}\omega^{ik}y_{k,j-i}(u)\otimes E_{ij},
\end{align*}
which shows that
$$s_{ij}(u)=\sum_{k\in\mathbb{Z}_N}\omega^{ik}y_{k,j-i}.$$
\vskip.1in
Since $A_{k,l}$ and $\frac{\omega^{kl}}{N}A_{-k,-l}$ are dual bases of $\mathfrak{gl}_N$, it follows that:
\begin{align*}
&y_{kl}(u)=(S(u)|A_{-k,-l})\frac{\omega^{kl}}{N}\\
&~~~~~~~~~=\sum_{i,j}\frac{\omega^{kl}}{N}s_{ij}(u)(E_{ij}|A_{-k,-l})\\
&~~~~~~~~~=\sum_{i,j}\frac{\omega^{k(l-j)}}{N}\delta_{i,j-l}s_{ij}(u)\\
&~~~~~~~~~=\sum_{i}\frac{\omega^{-ki}}{N}s_{i,i+l}(u)
\end{align*}
 \end{proof}
From the Theorem \ref{thm2}, we know that
as in the case $Y(\mathfrak{gl}_N)$, the principal generators $y_{ij}(u)$ of $Y(\mathfrak{g}_N)$ are actually  obtained from the Fourier transform of the sequence $\{s_{k,k+j}(u)\}$, where $k\in \mathbb{Z}_N$.
$$y_{ij}(u)=\mathfrak{F}^{-1}(\{s_{k,k+j}(u)\})(i).$$
\vskip.2in
 By Theorem \ref{thm2} we can get the following algebra homomorphism between the twisted Yangian  $Y(\mathfrak{g}_N)$ and the universal enveloping algebra $U(\mathfrak{g}_N)$.
\vskip.2in
\begin{theorem}
The assignment
 $$y_{kl}(u)\mapsto\frac{1}{N}\delta_{k,0}\delta_{l,0}+\frac{(u\pm\frac{1}{2})^{-1}}{N}B_{-k,l}$$
gives a homomorphism between $Y(\mathfrak{g}_N)$ and $U(\mathfrak{g}_N)$. Moreover, the assignment
$$B_{ij}\mapsto Ny^{(1)}_{-i,j}$$
defines an embedding $U(\mathfrak{g}_N)\hookrightarrow Y(\mathfrak{g}_N)$.
\end{theorem}
\vskip.2in
\begin{proof}
From Theorem \ref{thm2}, it follows that the mapping
$$y_{kl}(u)\mapsto \sum_{i\in
\mathbb{Z}_N}\frac{\omega^{-ki}}{N}s_{i,i+l}(u)$$ is an isomorphism
of the two presentations of $Y(\mathfrak{g}_N)$. And from the
proposition \ref{proposition3.3}, we have that the assignment
$$s_{ij}(u)\mapsto \delta_{ij}+(u\pm\frac{1}{2})^{-1}F_{ij}$$
is an algebra homomorphism between $Y(\mathfrak{g}_N)$ and $U(\mathfrak{g}_N)$.
Then we get an algebra homomorphism between $Y(\mathfrak{g}_N)$ and $U(\mathfrak{g}_N)$:
\begin{align*}
&y_{kl}(u)\mapsto \sum_{i\in \mathbb{Z}_N} \frac{\omega^{-ki}}{N}(\delta_{l,0}+(u\pm\frac{1}{2})^{-1}F_{i,i+l})\\
&~~~~~~~~~=\frac{1}{N}\delta_{l,0}\sum_{i\in\mathbb{Z}_N}\omega^{-ki}+\frac{(u\pm\frac{1}{2})^{-1}}{N}\sum_{i\in \mathbb{Z}_N}\omega^{-ki}F_{i,i+l}\\
&~~~~~~~~~=\frac{1}{N}\delta_{l,0}\delta_{k,0}+\frac{(u\pm\frac{1}{2})^{-1}}{N}\sum_{i\in \mathbb{Z}_N}\omega^{-ki}F_{i,i+l}.
\end{align*}
Subsequently one has
\begin{align*}
&\sum_{i\in \mathbb{Z}_N}\omega^{-ki}F_{i,i+l}=\sum_{i\in\mathbb{Z}_N}\omega^{-ki}E_{i,i+l}-\sum_{i\in\mathbb{Z}_N}\omega^{-ki}E_{i,i+l}^t\\
&~~~~~~~~~~~~~~~~~~~~~~~~~~~~~~~~~~~~~~~~~~~~~ =A_{-k,l}-A_{-k,l}^t\\
\end{align*}
which is just $B_{-k,l}$. So the assignment
\begin{equation*}
y_{kl}(u)\mapsto\frac{1}{N}\delta_{k,0}\delta_{l,0}+\frac{(u+\frac{1}{2})^{-1}}{N}B_{-k,l}
\end{equation*}
does define an algebra homomorphism between $Y(\mathfrak{g}_N)$ and $U(\mathfrak{g}_N)$.

\vskip.1in
Similarly, we can show that the assignment $B_{ij}\mapsto Ny^{(1)}_{-i,j}$  defines an embedding $U(\mathfrak{g}_N)\hookrightarrow Y(\mathfrak{g}_N)$.

\end{proof}

\medskip

\centerline{\bf Acknowledgments}
NJ gratefully acknowledges the partial support of Max-Planck Institut f\"ur Mathematik in Bonn, Simons Foundation grant 198129, and NSFC grant 10728102 during this work.

\bibliographystyle{amsalpha}

\end{document}